\documentclass{article}

\usepackage{graphicx}
\usepackage{latexsym}
\usepackage{amssymb}
\usepackage{lineno}
\usepackage{enumerate}
\usepackage{amsmath}
\usepackage{amsthm}
\usepackage{color,soul}
\usepackage{authblk}
\usepackage{dirtytalk}

\def\spo{\mathring{{\rm s}}}

\newtheorem{theorem}{Theorem}[section]
\newtheorem{corollary}[theorem]{Corollary}
\newtheorem{lemma}[theorem]{Lemma}

\newtheorem{definition}[theorem]{Definition}
\newtheorem{example}[theorem]{Example}

\makeatletter
\newcommand{\fcolon}{%
  \mathrel{\mathpalette\fcolon@\relax}%
}
\newcommand{\fcolon@}[2]{%
  \sbox\z@{$\m@th#1:$}%
  \vbox to\ht\z@{%
    \hbox{$\m@th#1.$}%
    \vss
    \hbox{$\m@th#1.$}%
    \vss
    \hbox{$\m@th#1.$}%
  }%
}
\makeatother

\title{\bf Slow Coloring of \(3k\)-connected Graphs}

\author{{\Large Joan M. Morris \hspace{6mm} Gregory J. Puleo}\\ Department of Mathematics and Statistics \\ 221 Parker Hall, Auburn University \\ Auburn, AL 36849\\
{\tt jmh0127@auburn.edu \hspace{6mm} gjp0007@auburn.edu}}

\date{}

\begin{document}
\nocite{*}

\maketitle

\begin{abstract}
The \textit{slow coloring game} was introduced by Mahoney, Puleo, and West and it is played by two players, Lister and Painter, on a graph \(G\). In round \(i\), Lister marks a nonempty subset \(M\) of \(V(G)\). By doing this he scores \(|M|\) points. Painter responds by deleting a maximal independent subset of \(M\). This process continues until all vertices are deleted. Lister aims to maximize the score, while Painter aims to minimize it. The best score that both players can guarantee is called the \textit{slow coloring number} or \textit{sum-color cost} of \(G\), denoted \(\spo{(G)}\). 

Puleo and West found that for an \(n\)-vertex tree \(T\), the slow coloring number is at most \(\lfloor \frac{3n}{2} \rfloor\), and that the maximum can be reached when \(T\) contains a spanning forest with vertices of degree 1 or 3. 
This implies that every n-vertex graph \(G\) having a perfect matching satisfies \(\spo(G) \geq \lfloor{\frac{3n}{2}}\rfloor\). 
In this paper, we prove that for \(3k\)-connected graphs with \(|V(G)| \geq 4k\) and with a perfect matching the lower bound is higher:  \(\spo(G) \geq \frac{3n}{2} + k\).
\end{abstract}

\section{Introduction}
\label{S:1}

The \textit{slow coloring game} \cite{MPWSpo} is played between two players, Lister and Painter, on a graph \(G\).  In each round of the game, Lister marks a nonempty subset of the graph, which we'll call \(M\), and scores \(|M|\) points. Painter then chooses a maximal independent subset of \(M\) to delete. This process continues until all the vertices are deleted. Lister seeks to maximize the score and Painter seeks to minimize it. The best score that each player can guarantee is called the slow coloring number, or sum-color cost of \(G\), \(\spo(G)\). 
Slow coloring, also called online sum-paintability, is a recent problem that comes from a history of coloring parameters. These other variations of coloring can help us to better understand the slow-coloring game and sum-color cost. A proper coloring of a graph, \(G\), is an assignment of colors to the vertices of \(G\) such that adjacent vertices must get distinct colors. List coloring, introduced independently by Erdos-Rubin-Taylor \cite{ERT} and Vizing \cite{vizing-list, tuza-survey}, gives the graph a list assignment \(L\), such that each vertex \(v\) receives a list of \(L(v)\) available colors.

\begin{definition}
A graph, \(G\), is \textit{\(L\)-colorable} if it has a proper vertex coloring using the colors from the lists assigned by \(L(v)\).
\(G\) is \textit{ \(f\)-choosable} for a function \(f: V(G) \rightarrow \mathbb{N}\), if it is \(L\)-colorable for every list assignment \(L\) such that \(|L(v)| \geq f(v)\) for all \(v\). \(G\) is \textit{\(k\)-choosable} for an integer \(k\) if it is \(f\)-choosable when \(f(v)= k\) for all \(v\).  
\end{definition}

One can also consider other measures to assess a list coloring, such as the least sum or the average of the list sizes. For example, the sum-choosability, \(\chi_{SC}(G)\), introduced by Isaak is defined to be the minimum \(\sum(f(v))\) over all \(f\) such that \(G\) is \(f\)-choosable. \cite{Isaak}

We can modify the definition of list coloring by introducing an online factor: reveal the lists of vertices as a function of time: as time passes, more information is revealed. We can view this as a game with the same players and round structure described in relation to slow coloring. In round \(i\), Lister marks a subset of \(M\) vertices of the graph. We can view this marking as revealing all of the vertices with color \(i\) in their lists. Painter then chooses an independent subset of \(M\) to receive color \(i\). In comparison with slow coloring however, we score this game differently: Lister tries to maximize the amount of times that a vertex is chosen, thus revealing the entire hidden list before Painter colors it. 

We note here the connection to choosability, where we're concerned with the largest list on any vertex. For a function \(f\) determining the list sizes for the vertices, Lister wins if some vertex \(v\) is marked more than \(f(v)\) times. Painter wins by coloring all the vertices before this happens. Thus, Painter wins the \(f\)-painting game by preventing a vertex \(v\) from being marked more than \(f(v)\) times. If Painter can win, then the graph is said to be \(f\)-paintable.  A graph is \(k\)-paintable if it is \(f\)-paintable for the function \(f(v) = k\) for all \(v\), and the \textit{paintability} of a graph is the least such \(k\). Paintability was independently introduced by Schauz \cite{SchauzPaintability} and Zhu \cite{ZhuPaintability} Just as in choosability, we can study the least sum, or average, of this property-- the sum-paintability of \(G\), studied first by Carraher, Mahoney, Puleo, and West \cite{CMPW2013}. Denoted \(\chi_{SP}(G)\), the sum-paintability of a graph \(G\) is the least value of \(\sum(f(v))\) such that \(G\) is \(f\)-paintable.

To clarify this process, consider an example on a cycle of length 5, denoted \(C_5\), with vertices labeled \(v_1, \dots, v_5\).

\begin{example}

We define a function \(f\) such that for each \(v \in V(C_5)\), the size of \(v\)'s list is \(f(v) = 2\). 

 \begin{center}
    \includegraphics[scale=0.3]{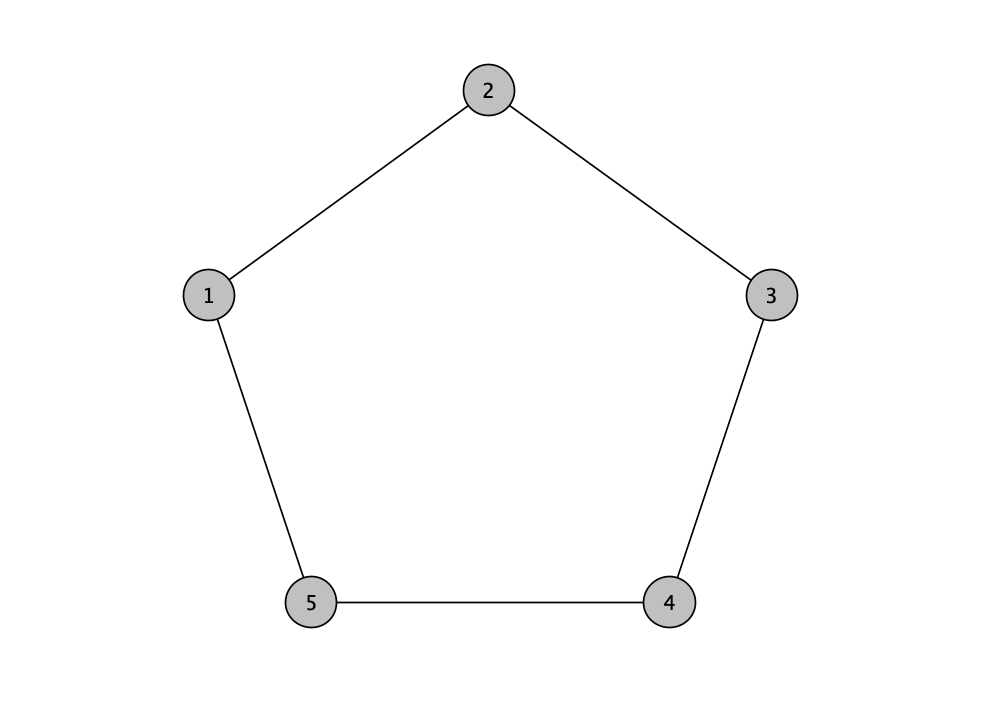}
    \label{C5}
    \end{center}

In the first move, Lister selects all vertices of \(C_5\), \(v_1, \dots, v_5\), revealing that they have color \(1\) in their lists. Painter responds by selecting an independent set to color as color \(1\). Without loss of generality we'll say he chooses \(v_1\) and \(v_3\). In the second round, Lister chooses \(v_4, v_5\), revealing color \(2\) in their lists. Since these two vertices are adjacent, Painter can only select one. Say he chooses \(v_5\) and gives it color \(2\). In the last round, Lister chooses \(v_2, v_5\), revealing color \(3\) in their lists. This forces \(v_5\) to have the list \(L(v_5) = \{1, 2, 3\}\). Since \(|L(v)| > 2 = f(v_5)\), this shows that \(C_5\) is not \(2\)-paintable. Executing the same process with \(f(v) = 3\) for all \(v\) shows that \(C_5\) is \(3\)-paintable.

\begin{figure}[htbp]
    \centering
    \includegraphics[scale = 0.3]{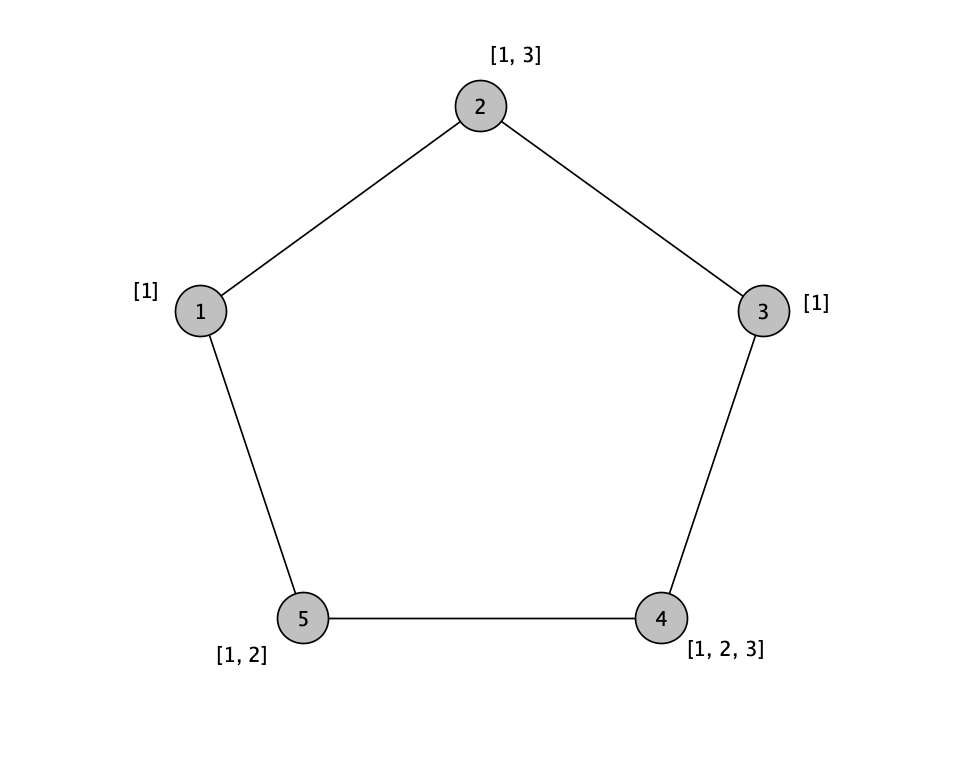}
    \caption{\(C_5\) with revealed lists}
    \label{C5Lists}
\end{figure}

\end{example}

Since paintability deals with how many times a vertex is marked, and the size of its list, we are not focused on the specific colors assigned to the vertices. Hence, as observed in \cite{CMPW2013}, we can view paintability in the following way: Painter allots tokens to the vertices of \(G\), according to a function\(f(v)\), corresponding to the size of their lists. Every time Lister marks a vertex, a token is removed. When all the tokens of a vertex have been used, then the vertex has been marked \(f(v)\) times. If a vertex is marked more than \(f(v)\) times (having no tokens left to \say{pay}), Lister wins the game. In this fashion, sum-paintability is the least amount of total tokens used. If we continue in this direction and consider  sum-paintability in an online progression, we arrive at slow coloring. 

In slow coloring, rather than assigning tokens beforehand according to \(f\), Painter can distribute tokens to the vertices as the game progresses. This allows Painter to reserve tokens, and use them as needed, perhaps on especially difficult vertices. Thus, we can see that \(\spo(G) \leq  \chi_{SP}(G)\), since Painter can always play according to the function defined by \(\chi_{SP}(G)\). Here again, we can see that the specific color of the vertices marked does not affect the parameter, and for each round \(i\), we can use a different color \(i\). Thus deleting a vertex in round \(i\), as discussed earlier, is a model for assigning it color \(i\).

\subsection{History}

Slow coloring was first introduced by Mahoney, Puleo, and West in 2017  \cite{MPWSpo}. They provided a general upper and lower bound for the slow coloring number \(\spo\), which is dependent on the graph's number of vertices and its independence number, which is the largest subset of vertices in \(G\) such that none of the vertices are adjacent: 

\[
\displaystyle \frac{|V(G)|}{2\alpha(G)} + \frac{1}{2} \leq \displaystyle \frac{\spo(G)}{|V(G)|} \leq \max \left \{\frac{|V(H)|}{\alpha(H)} \colon H \subset G \right\}
\]

They also produced results for specific cases, including for graphs with independence number  two, for \(n\)-vertex trees and for complete bipartite graphs. In 2018, Gutowski, et al, \cite{GSpoSparse} studied the slow coloring number of several classes of sparse graphs including \(k\)-degenerate, acyclically \(k\)-colorable, planar, and outerplanar graphs. 

Around the same time, Puleo and West \cite{PWTrees} published results studying slow coloring on trees. They developed an algorithm to compute the slow coloring number for a tree and produced results characterizing \(n\)-vertex trees with the largest and smallest values.  They proved two theorems in particular, the second of which will be useful for our results: 

\begin{theorem} \label{Puleo1} \cite{PWTrees}
For every \(n\)-vertex tree \(T\), 

\[
n + \sqrt{2n} \approx n + u_{n-1} = \spo(K_{1, n-1}) \leq \spo(T) \leq \spo(P_n) = \lfloor \frac{3n}{2} \rfloor
\]
where \(u_r = \max \{k : t_k \geq r \} \)  for \(t_k = \binom{k+1}{2} \), \(k, r \in \mathbb{N}\).

\end{theorem}

\begin{theorem} \label{PuleoSP} \cite{PWTrees}
If \(T\) is an n-vertex forest, then \(\spo(T) = \lfloor \frac{3n}{2} \rfloor\) if and only if \(T\) contains a spanning forest in which every vertex has degree 1 or 3, except for one vertex of degree 0 or 6 when \(n\) is odd. 
\end{theorem}

A natural corollary of this follows by using the fact that any graph with a perfect matching has one of these spanning forests as a subgraph. Since more edges would only push the sum-color cost higher, this becomes a lower bound for classes of graphs with a perfect matching. If \(G\) is a graph with a perfect matching and  \(|V(G)| =n\), we have \(\spo(G_\sigma) \geq \frac{3n}{2}\).

\subsection{Main Result}

In this paper, we use a Lister strategy to guarantee a higher bound for sufficiently large \(3k\)-connected graphs with perfect matchings: 

\begin{theorem} \label{MainResult1}
Let \(G\) be a \(3k\)-connected graph with \(|V(G)| \geq 4k\) and with a perfect matching. Then the slow coloring number of \(G\) is bounded from below by
\[
\spo(G_\sigma) \geq \displaystyle \frac{3n}{2} +k
\]
\end{theorem}

This bound is sharp in the \(k=1\) case, and cannot be sharp in the \(k>1\) case. 

\vspace{0.5in}
%%%%%%%%%%%%%%%%%%%%%%%%%%%%%%%%%
%%%%%%%%%%%%%%%%%%%%%%%%%%%%%%%%%
%%%%%%%%%%%%%%%%%%%%%%%%%%%%%%%%%

\section{Main Results}
\label{S:2}

Throughout this section let \(G\) be a \(3k\)-connected graph with \(|V(G)| \geq 4k\) and with a set of edges, \(P\), which induces a perfect matching. Lister begins by choosing \(4k\) vertices using \(2k\) pairs in \(P\).  Painter replies by deleting an independent subset of the marked vertices, \(D\). Note that \(|D| \leq 2k\), since \(D\) contains most one vertex in each pair in \(P\). Let \(G^- = G - D\).

\begin{lemma} \label{kconlem}
\(G^-\) is \(k\)-connected. 
\end{lemma}

\begin{proof}
We need to show that given any \(S \subset V(G^-)\) with \(|S| \leq k-1\), \(G^- - S\) is connected. Let \(D^* = D \cup S\); then \(|D^*| \leq 2k + (k - 1) = 3k - 1\). Since \(G\) is \(3k\)-connected, \(G - D^*\) is connected. Since 
\[
G - D^* = G - D - S = G^- - S,
\]

\(G^- - S\) is connected so \(G^-\) is \(k\)-connected.
\end{proof}

Consider the vertices in \(G\) that are joined to vertices in \(D\) by matching edges; call them \(\beta\)-vertices and the set of such vertices \(D'\). So \(G^-\) contains \(|D| = |D'| \leq 2k\) \(\beta\)-vertices. We consider separately the cases when \(|D'|\) is even and when \(|D'|\) is odd.

For two sets \(A, B\) the \textit{symmetric difference} of \(A\) and \(B\) is defined to be the set of elements, each of which is in exactly one of \(A, B\); denoted by \(A \oplus B \).

\begin{lemma} \label{3kspan}
If \(|D'|= 2k^*\) , then \(\spo(G^-) \geq \frac{3}{2}n-3k^*\). 
\end{lemma}

\begin{proof}

Recall that \(G^-\) contains \(|D'| \leq 2k\) \(\beta\)-vertices. Arbitrarily partition \(D'\) into two sets  \(A\) and \(B\), each of size \(k^*\). By Lemma~\ref{kconlem}, \(G^-\) is \(k\)-connected. Thus, by Menger's Theorem \cite{West} and since \(k^* \leq k\), there exist \(k^*\) vertex-disjoint paths between vertices in \(A\) and vertices in \(B\), say \(P_1, P_2, \dots, P_{k^*}\). Let \(P = P_1 + P_2 + \dots + P_{k^*}.\) 

Construct a spanning forest \(F\) as follows: Let \(F_0\) be the set of all matching edges remaining in \(G^-\).  Define \(F_1\) by: 
\[
F_1 = F_0 \oplus E(P).
\]

The following shows that each vertex in \(F_1\) has degree \(1\) or \(3\). First consider a \(\beta\)-vertex, \(v_\beta\). Since in \(G^-\) \(\beta\)-vertices are not incident to a matching edge, \(d_{F_0}(v_\beta) = 0\). Since they are an endpoint of a path in \(P\), \(d_P(v_\beta) = 1\).  Thus \(d_{F_1}(v_\beta) =1\).

Now consider all the remaining vertices. Let \(v\) be a vertex that is not a \(\beta\)-vertex. Since \(v\) retains its matching edge in \(G^-\), \(d_{F_0}(v) = 1\). Also, \(d_P(v)\) is even: if the vertex is disjoint from \(P\), then \(d_P(v) = 0\); if the vertex belongs to a path in \(P\), then \(d_P(v) = 2\). Since \(P_1, P_2, \dots, P_k\) are vertex-disjoint, \(v\) can belong to at most one such path. Thus for non-\(\beta\)-vertices, \(d_{F_1}(v)\) is either \(1\) or \(3\). Hence every vertex has degree \(1\) or \(3\) in \(F_1\)

Finally, form \(F\) from \(F_1\), by iteratively removing arbitrarily chosen cycles until no cycle remains. Since the removal of any cycle does not affect the fact that all vertices have odd degree, \(F\) is a spanning forest in \(G^-\) with all vertices of degree \(1\) or \(3\) as required. Since \(|V(G^-)| = n-2k^*\), by \cite{PWTrees}, we have 
\[
\spo(G^-) \geq \frac{3}{2}(n-2k^*) \geq \frac{3}{2}n-3k^*.
\]
\end{proof}

\begin{lemma} \label{spoBodd}
If \(|D'| = 2k^*+1\) , then \(\spo(G^-) > \frac{3}{2}n-3k^*.\)
\end{lemma}

\begin{proof}
The number of \(\beta\)-vertices can only be odd if \(|D| < 2k\). Arbitrarily choose a \(\beta\)-vertex, \(v_{\beta}\), and divide \(G^-\) into two disjoint subgraphs: \(G_1^- = G^- - v_{\beta}\) and \(G_2^-\) being the isolated vertex \(v_{\beta}\). 
Notice that \(G_1^-\) is \(k\)-connected because for the number of \(\beta\)-vertices to be odd, we must have \(|D| \leq 2k-1\), so \(G_1^-\) is formed by removing at most \(2k\) vertices from the \(3k\)-connected graph \(G\).

Now that \(v_{\beta}\) has been removed, we define \(D^-\) as the set of remaining \(\beta\)-vertices in \(G_1^-\). Note that \(|D^-| = 2k^*\) since \(|D'| = 2k^*+1\). 
 By the argument we used in Lemma~\ref{3kspan}, we can say that  \(\spo(G_1^-) \geq \frac{3}{2}(n-2k^*) = \frac{3}{2}n-3k^*.\)

Since \(G_2^-\) consists of a single vertex, \(\spo(G_2^-) = 1\). By \cite{MPWSpo}, \(\spo(G^-) \geq \spo(G_1^-) + \spo(G_2^-)\), so 

\[
\spo(G^-) \geq \frac{3}{2} n-3k^* + 1 > \frac{3}{2}n -3k^*.
\]
\end{proof}

\begin{corollary} \label{spocombo}
\(\spo{(G^-)} \geq \frac{3}{2}n - 3k \) 
\end{corollary}

\begin{proof}
From Lemma~\ref{3kspan} and Lemma~\ref{spoBodd}, since \(k^* \leq k\), 

\[\spo(G^-) \geq \frac{3}{2}n - 3k^* \geq \frac{3}{2}n - 3k.\]   
\end{proof}

\begin{theorem} \label{MainResult}
Let \(G\) be a \(3k\)-connected graph with \(|V(G)| \geq 4k\) and with a perfect matching. Then the slow coloring number of \(G\) is bounded from below by
\[ \spo(G) \geq \frac{3n}{2} + k.\]
\end{theorem}

\begin{proof}
Suppose Lister begins by choosing \(4k\) vertices -- \(2k\) matching pairs. Painter replies by deleting an independent subset of the marked vertices, \(D\); so \(|D| \leq 2k\).  Note that \(G^-\) has at least \(n-2k\) vertices.

By Corollary~\ref{spocombo}, we know \(\spo(G^-) \geq \frac{3}{2}n-3k\). 

Playing an optimal strategy on \(G^-\),  the final score \(\spo\) achieved by Lister by choosing the \(4k\) vertices against this reply is bounded from below by

\[
\spo(G)  \geq \left( \frac{3}{2}n - 3k \right) +4k  = \frac{3}{2} n +k 
\]

Since this bound holds no matter which vertices Painter deletes in response to our Lister's choice, we conclude that 

\[
\spo(G)  \geq  \frac{3}{2} n +k.
\]

\end{proof}

\begin{corollary}
The inequality in Theorem~\ref{MainResult} is sharp when \(k=1\), but cannot be sharp for \(k>1\). 
\end{corollary}

\begin{proof}
We begin by noticing that for equality to be reached, two conditions must be satisfied. First, \(|D| = 2k\). This is achieved if it is possible for Painter to delete one vertex from each matching pair -- that is, that one vertex from each matching pair can form an independent set. Secondly, the resulting graph \(G^-\) must satisfy Theorem~\ref{PuleoSP}. This theorem requires \(G^-\) to be a forest, and to contain a spanning forest in which every vertex has degree 1 or 3. The following graph, called the triangular prism graph, satisfies both conditions.

    \begin{center}
    \includegraphics[scale=0.3]{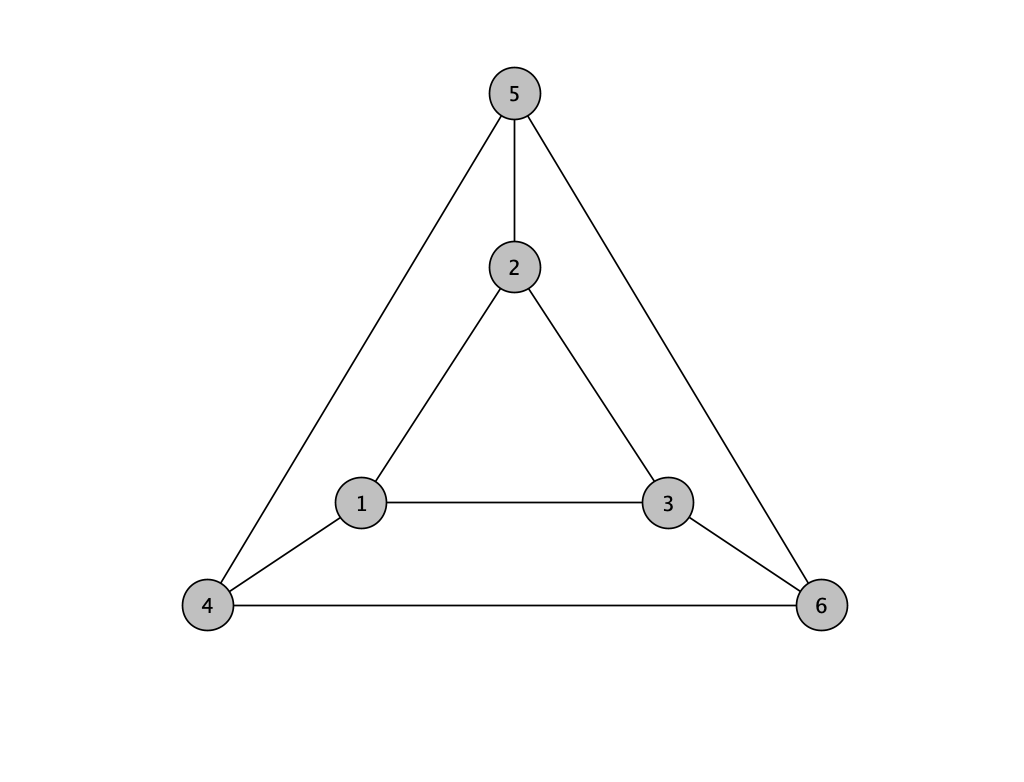}
    \label{PrismPic}
    \end{center}

Let \(G\) be the triangular prism graph. \(G\) is 3-connected, has \(6>4\) vertices, and has a perfect matching given by \(\{(1,4), (2,5), (3,6)\}\). Thus \(G\) fulfills the requirements of Theorem~\ref{MainResult}. We suppose Lister selects all 6 vertices, and Painter deletes one vertex from two matching pairs, namely \(3\) and \(4\).

    \begin{center}
    \includegraphics[scale=0.3]{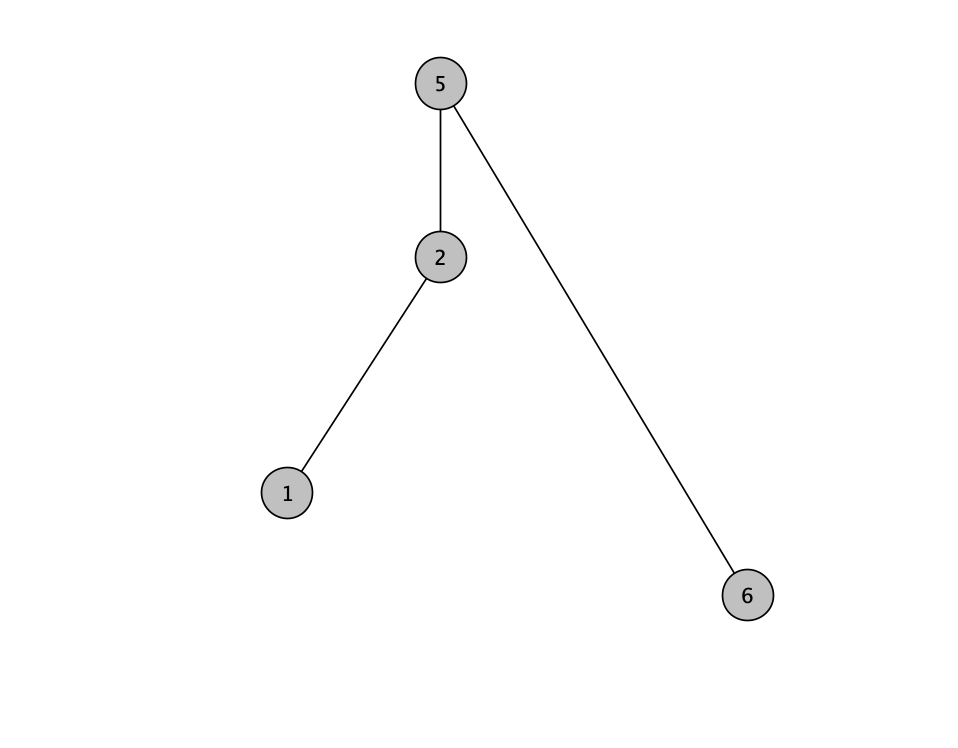}
    \label{PrismMinus}
    \end{center}

The resulting graph \(G^-\) is a forest, and contains a spanning forest in which the vertices have degree 1 and 3. Therefore by Theorem~\ref{PuleoSP}, we have that \(\spo{(G^-)} = \frac{3}{2}(4) = 6 \). Adding Lister's move to get the slow coloring number for \(G\), we have \(\spo(G) = 6+4 = 10 = \frac{3}{2}(6) + 1\).

To see that the inequality is not sharp for \(k>1\), consider Whitney's Theorem \cite{West} which states that \(\kappa(G) \leq \delta(G)\) where \(\kappa(G)\) is the vertex-connectivity and \(\delta(G)\) is the minimal degree of \(G\). If \(G\) is \(3k\)-connected then \(\delta(G) \geq 3k\). We delete at most \(2k\) vertices, so we have \(\delta(G^-) \geq 3k -2k =k\). But \(k \geq 2\), so \(\delta(G^-) \geq 2\). Therefore, \(G^-\) must have a cycle and doesn't fulfill the requirements of Theorem~\ref{PuleoSP}. Thus for \(k>1\), we can conclude \(\spo(G) > \frac{3}{2}n +k\). 
\end{proof}

\bibliographystyle{unsrt}
\bibliography{Spo3kJournal}

\end{document}